\documentclass[a4paper]{amsart}
\usepackage{amsmath}
\usepackage{amssymb}
\usepackage{amsthm} 
\usepackage{mathtools}
\usepackage{graphicx}
\usepackage{xcolor}
\usepackage{multirow}
\usepackage{booktabs}
\usepackage{calc}
\usepackage{array}
\usepackage{hyperref}
\usepackage[noadjust]{cite}
\usepackage{bbm}
\usepackage{setspace}
\usepackage{fancyhdr}

\makeatletter
\@addtoreset{footnote}{section}
\makeatother

\newtheorem{theorem}{Theorem}[section]
\newtheorem{lemma}[theorem]{Lemma}
\newtheorem{proposition}[theorem]{Proposition}

\newtheorem{corollary}[theorem]{Corollary}
\newtheorem*{corollary*}{Corollary}
\newtheorem{definition}[theorem]{Definition}
\newtheorem*{definition*}{Definition}

\theoremstyle{remark}
\newtheorem{remark}[theorem]{Remark}
\newtheorem*{example*}{Example}
\newtheorem*{notation*}{Notation}
\newtheorem*{remark*}{Remark}
\newtheorem{example}[theorem]{Example}

\newcommand{\gap}{\mbox{ }}

\newcommand{\R}{\mathbb{R}}
\newcommand{\N}{\mathbb{N}}

\newcommand{\B}{\mathcal{B}}
\newcommand{\Z}{\mathbb{Z}}
\newcommand{\C}{\mathbb{C}}

\newcommand{\V}{\mathcal{V}}

\newcommand{\M}{\mathcal{M}}

\newcommand{\la}{\langle}
\newcommand{\ra}{\rangle}

\newcommand{\del}{\partial}

\newcommand{\ddc}{dd^c}

\DeclareMathOperator*{\Span}{span}

\newcommand{\Pbb}{\mathbb{P}}
\newcommand{\Lcal}{\mathcal{L}}

\newcommand{\dbar}{\bar\del}

\newcommand{\bs}{\backslash}

\newcommand{\lt}{\textsc{lt}}

\newcommand{\cC}{\mathcal{C}}
\newcommand{\bfI}{{\bf I}}
\newcommand{\bfV}{{\bf V}}
		
\begin{document}
\title[Transfinite Diameter for Affine Algebraic Varieties]{Relations between Transfinite Diameters on Affine Algebraic Varieties}
\author{Jesse Hart}
\email{jesse.hart@auckland.ac.nz}
\date{\today}
\begin{abstract}
Given a compact set $K$ one may define a transfinite diameter for $K$ via a limiting process involving maximising a Vandermonde determinant over $K$ with respect to a monomial basis. Different transfinite diameters may be obtained by using different polynomial bases in the Vandermonde determinant calculation. We show that if these bases are sufficiently similar that the transfinite diameter of $K$ is unchanged. Utilising this result we show that the transfinite diameters defined by Cox-Ma`u and Berman-Boucksom for algebraic varieties are equal. 
\end{abstract}
\maketitle
\section{Introduction}
The transfinite diameter of a compact set $K$ is an interesting constant within pluripotential theory due to the fact that it is related to other seemingly unrelated constants such as Chebyshev constants and logarithmic capacity. In $\C$ these relationships have been well understood for some time while it wasn't until Rumely \cite{Rumely07} and Berman-Boucksom \cite{Berman08, Berman10, Berman11} that these relationships were well understood in $\C^M$. Essential to the definition of a transfinite diameter is the evaluation of a Vandermonde determinant with respect to some polynomial basis. In $\C^M$ this basis is taken to be the monomial basis while in the setting of a complex manifold Berman-Boucksom chose a $\mu$-orthonormal basis where $\mu$ is a probability measure.\\\gap\\
Studying transfinite diameters in the setting of an algebraic variety $\V$ is in some sense a midway point between $\C^M$ and a complex manifold. The variety inherits enough structure from $\C^M$ in order to show convergence of the transfinite diameter limit in an analogous way, but can also be viewed as a complex manifold (possibly with singularities)  so one could use a complex geometric approach to show convergence. Both these approaches have been studied, the former by Cox-Ma`u \cite{Mau11, CoxMau15} and the latter by Berman-Boucksom \cite{Berman08, Berman10}. While these transfinite diameters are defined in similar ways the precise relationship between the two approaches has not yet been clarified.\\\gap\\
This paper establishes that with an appropriate probability measure these approaches yield the same transfinite diameter under mild geometric hypothesis. The first step towards showing this is establishing a general technical theorem relating transfinite diameters defined by bases which are sufficiently similar, which we call {\em compliant bases} (Definition \ref{compliant}). 
\begin{theorem}[Theorem \ref{mainthm}]
Let $\V$ be an affine algebraic variety of pure dimension $M$, $K\subset \V$ a compact set, and suppose that $\B$ and $\cC$ are compliant graded bases for $\C[\V]$. If the limit
\[\log d^\B(K)=\lim_{k\rightarrow\infty} \frac{1}{l_k} \log\|VDM[\B_k]\|_{L^\infty(K)}\]
exists, then
\[\log d^{\cC}(K)=\lim_{k\rightarrow\infty} \frac{1}{l_k} \log\|VDM[\cC_k]\|_{L^\infty(K)}\]
exists. Moreover, $d^\B(K) = d^{\cC}(K)$. 
\end{theorem}
Equipped with this result we show the desired equality of approaches by showing that the basis used in each approach is compliant with the usual monomial basis the algebraic variety. For the Cox-ma`u approach this is done in Proposition \ref{cmmcompliant} while the Berman-Boucksom approach is done in Proposition \ref{bbcompliant} where the `appropriate' probability measure is $\nu=(\ddc V_{T_\V})^M$ (Proposition \ref{cxon}). With these propositions the following corollary is immediate.
\begin{corollary}[Corollary \ref{propcmbb}]
Let $\V$ be an affine algebraic variety with distinct intersections with infinity and Noether presentation $(x,y)$. If $K\subset \V$ is a compact set then
\[d^{cm}(K) = d^{bb}_\nu(K).\]
\end{corollary}
The work of Berman-Boucksom uses a slightly different definition to what has been standard. The relation between the two is a constant depending only on the dimension of the variety. 
\begin{corollary}[Corollary \ref{bbrelated}]
Let $\V$ be an algebraic variety, $K\subset\V$ a compact set, and $\nu$ a probability measure on $\V$. If $S_k$ is an orthonormal basis for $\C_{k}[\V]$ then
\[\frac{M+1}{M}\log d^{bb}_\nu(K) = \lim_{k\rightarrow\infty}\frac{1}{kN_k}\log\|VDM[S_k]\|_{L^{\infty}(K)}.\]
\end{corollary}
\noindent The limit on the right hand side is the original formulation of Berman-Boucksom. \\\gap\\The paper is structured as follows: Section 2 is a brief overview of the necessary preliminary information for this paper, Section 3 contains the proof of the main technical theorem in the paper, Section 4 utilises our main theorem to relate $d^{cm}(K)$ and $d^{bb}_\nu(K)$, Section 5 shows how $d^{bb}_\nu(K)$ is related to the original formulation of Berman-Boucksom, and Section 6 is a worked example illustrating the equality. 
%Section 2
\subsection{Notation}
Let $\V$ be an affine algebraic variety of pure dimension $M$. We employ the following notational conventions.
\begin{itemize}
\item{$\Z_{\geq 0}$ is the set of nonnegative integers, and $\Z^M_{\geq0}$ is the set of $M$-tuples of nonnegative integers. }
\item{If a polynomial ordering is not stated, it is assumed that the ordering used is grevlex. Precisely, if $\alpha,\beta\in\Z^N_{\geq 0}$ then $\alpha \succ_{grevlex} \beta$ if
\begin{itemize}
\item[(i)]{$|\alpha|>|\beta|$, or;}
\item[(ii)]{$|\alpha|=|\beta|$ and $\alpha_{N-j}>\beta_{N-j}$ for the first index $j\in\Z_{\geq 0}$ for which $\alpha_{N-j}\neq\beta_{N-j}$.}
\end{itemize}}
\item{$\bfI(\V)$ denotes the ideal generated by $\V$. If $p_1,...,p_s$ is a collection of polynomials then $\bfV(p_1,...,p_s)$ is the variety defined by the common zero set of $p_1,...,p_s$.}
\item{Given a polynomial $p$ and a monomial ordering, $\lt(p)$ is the term of the polynomial with maximum degree with respect to this ordering. The ideal $\la \lt(I(\V))\ra$ is the ideal generated by the terms \[\lt(I(\V)) = \{cz^\alpha : \text{these exists }p\in I(\V)\text{ such that }\lt(p)=cz^\alpha, \alpha\in\Z^{N}\}.\] }
\item{Consider the ring of equivalence classes $\C[z]/\bfI(\V)$ and choose some monomial ordering. A canonical representative from an equivalence class $[p]$ is the unique element $q$ such that no term of $q$ lies in $\la \lt(\bfI(\V)) \ra$, we call such an element a {\em normal form} (Definition \ref{normalform}). $\C[\V]$ is the set of all normal forms on $\V$ and we call this set the polynomials on $\V$.}
\item{$\C_{=k}[\V]$ are the polynomials in $\C[\V]$ of degree exactly $k$, and $\C_{k}[\V]$ are the polynomials in $\C[\V]$ of degree at most $k$.}
\item{$N_{=k}$ is the number of monomials in $\C_{=k}[\V]$ and $N_{k}$ the number of monomials in $\C_{k}[\V]$. Equivalently, $N_{=k} = \dim(\C_{=k}[\V])$ and $N_k=\dim(\C_k[\V])$.}
\item{$l_{k}=\sum_{j=1}^k jN_{=j}$, i.e.\ is the sum of the degrees of the monomials of degree at most $k$.}
\item{$N_{=k}^x$ is the number of monomials in $\C_k[x]$, and $N^x_{k}$ and $l^x_{k}$ are defined similarly.}
\item{$\M[x]$ denotes the usual monomial basis for $\C[x]$, $x\in\C^M$, and $\M[\V]$ denotes the usual monomial basis for $\C[\V]$. $\M_k[x]$ and $\M_k[\V]$ are the elements in $\M[x]$ and $\M[\V]$ respectively of degree at most $k$.}
\item{If $A$ is a collection of multi-indices in $\Z^s_{\geq 0}$ and $\C[\V] = \bigoplus_{\alpha\in A} y^{\alpha}\C[x]$ then $a=\max_{\alpha\in A}|\alpha|$ and $n=|A|$ where $|\alpha|=\alpha_1+...+\alpha_{s}$.}
\item{The operator $\ddc = \frac{i}{\pi} \del\dbar$.}
\end{itemize}
\section{Preliminaries}
\subsection{Transfinite Diameter}
Consider an affine algebraic variety $\V$ of pure dimension $M$. Let $\cC$ be a polynomial basis for $\C[\V]$ and suppose that $\cC_{k}$ are the elements of $\cC$ of degree at most $k$. We say that $\cC$ is a graded basis if $\cC_k$ is a basis for $\C_k[\V]$ for each $k$. Suppose that $\cC$ is a graded polynomial basis and let $\{e_i(z)\}_{i=1}^{s_k}$ be the elements in $\cC_k$. We use the notation
\[VDM_{\zeta_1,...,\zeta_{s_k}}[\cC_k] = \det\begin{pmatrix} e_1(\zeta_1) & e_1(\zeta_2) & ... & e_{1}(\zeta_{s_k}) \\  
 e_2(\zeta_1) & e_2(\zeta_2) & ... & e_{2}(\zeta_{s_k}) \\  
  \vdots & \vdots & \ddots & \vdots \\
  e_{s_k}(\zeta_1) & e_{s_k}(\zeta_2) & ... & e_{s_k}(\zeta_{s_k})
  \end{pmatrix}  \]
  and if $K\subset\V$ is a compact set,
\[\|VDM[\cC_k]\|_{L^\infty(K)} = \max_{\zeta_1,...,\zeta_{s_k}\in K}\left|VDM_{\zeta_1,...,\zeta_{s_k}}[\cC_k]\right|.\]
We occasionally write $VDM_{\zeta_1,...,\zeta_{s_k}}[\cC_k] = VDM[\cC_k]$ when the points $\zeta_1,...,\zeta_{s_k}$ are irrelevant within the context. The transfinite diameter of $K$ (with respect to the basis $\cC$) is the limit
\[d^\cC(K)=\lim_{k\rightarrow\infty} \left(\|VDM[\cC_k]\|_{L^\infty(K)}\right)^{1/l_k}, \]
if the limit exists. The transfinite diameter was introduced by Fekete-Szeg{\"o} for $K\subset\C$ as early as the 1920's and a comprehensive study of this and associated results can be found in \cite{Ransford}. It is a result of Zakharjuta \cite{Zakharjuta75} that the above limit exists when $\V=\C^N$ and $\cC$ is the ordinary monomial basis. 
\subsection{Cox-ma`u Transfinite Diameter}
The approach of Cox-ma`u \cite{Mau11,CoxMau15} is motivated by classical methods i.e.\ those of Zakharjuta \cite{Zakharjuta75}. Essential to this work is the creation of polynomials which have a sub-multiplicative property for which techniques of Zakharjuta can be used on to show the convergence of the transfinite diameter limit. 
%\begin{definition}
%Let $z=(x,y)$ where $x\in\C^M$, $y\in\C^{N-M}$ and $\V$ be an affine algebraic variety. We say that $\C[x]$ is a Noether normalisation for $\C[\V]$ if the following conditions are met:
%\begin{itemize}
%\item{$\C[x]$ is a subring of $\C[\V]$, equivalently, $\C[x]\cap \bfI(\V)=\{0\}$.}
%\item{There are finitely many elements $\{y^\alpha : \alpha\in A\subset\Z^{N-M}_{\geq0}\}$ such that any element $f\in\C[\V]$ can be expressed as $f=\sum g_\alpha(x)y^\alpha$ where $g_\alpha\in\C[x]$ for each $\alpha$. In this case we also say that
%\[\C[\V]=\sum_{\alpha\in A}y^\alpha\C[x].\] }
%\end{itemize}
%\end{definition}
\begin{definition}\label{noether}
Let $z=(x,y)$ where $x\in\C^M$, $y\in\C^{N-M}$ and $\V$ be an affine algebraic variety. We say that $\C[x]$ is a Noether normalisation for $\C[\V]$ if the following conditions are met.
\begin{itemize}
\item[(i)]{$\C[x]\cap \bfI(\V)=\{0\}$.}
\item[(ii)]{$\C[\V]$ is finite over $\C[x]$.}
\item[(iii)]{For each $1\leq i\leq N-M$, there is a $g_i\in I(\V)$ and $m_i\in\Z_{\geq 0}$ such that $\lt(g_i) = y_i^{m_i}$.}
\end{itemize}
\end{definition}
Given an algebraic variety $\V$, Theorem 2.1 \cite{CoxMau15} guarantees that one can always find a complex linear change of coordinates so that the three properties in Definition \ref{noether} hold. The condition (iii) is not usually assumed for a Noether normalisation as it is a consequence when using the lex monomial ordering. Since the motivating monomial ordering for the work of Cox-ma`u is the grevlex ordering we make this condition a requirement.
\begin{definition}
Let $\V\subset\C^N$ be an $M$ dimensional affine algebraic variety of pure dimension $M$ with $d$ sheets and affine coordinates $(x,y), x\in\C^M, y\in\C^{N-M}$. We say that  $\V$ has distinct intersections with infinity if it satisfies the following properties. 
\begin{itemize}
\item[(i)]{$\C[x]$ is a Noether normalisation for $\C[\V]$.}
\item[(ii)]{If $P=\bfV(t, x_1,..., x_{M-1})\subset\Pbb^N$ and 
\[\V_{\Pbb} = \{[t:x:y]\in\Pbb^N : t^{\deg p}p\left(\frac{x}{t},\frac{y}{t}\right) = 0,\, \forall p\in \bfI(\V)\}\]
where $x=(x_1,...,x_{M})$ and $y=(y_1,...,y_{N-M})$. Then $P\cap \V_\Pbb$ consists of $d$ points $\lambda_1,...,\lambda_d$.}
\item[(iii)]{If $\lambda_i = [0:...:0:\lambda_{ix_M}:\lambda_{iy_{1}}...:\lambda_{iy_{N-M}}]$ then $\lambda_{ix_M}\neq 0$ for each $i$.}
\end{itemize}
\end{definition}
\begin{example}\label{example2.2}
Consider $\V=\bfV(x^2-y^2-1)$ which is a 1 dimensional affine algebraic variety with two sheets.  Since $y^2=x^2-1$ it follows that
\[\C[\V] = y\C[x] + \C[x]\]
so $\C[x]$ is a Noether normalisation for $\C[\V]$. $\V_\Pbb = \bfV(x^2-y^2-t^2)$ so \[P\cap \V_\Pbb = \{[t:x:y]\in\Pbb^2 : t=0, x^2-y^2=0\}=\{\{[0:1:1], [0:1:-1]\}.\] It follows that $\V$ has distinct intersections with infinity. 
\end{example}
\begin{example}
Consider $\V=\bfV((x+y)^2+x+y-1)$ which is a 1 dimensional affine algebraic variety with two sheets. Then 
\[\C[V] = y\C[x]+\C[x]\]
is a Noether normalisation for $\C[\V]$. $\V_\Pbb = \bfV((x+y)^2+tx+ty-t^2)$ so \[P\cap \V_\Pbb = \{[t:x:y]\subset\Pbb^2 : t=0, (x+y)^2 = 0\} = \{[0:1:-1]\}.\]
It follows that $\V$ does not have distinct intersections with infinity.
\end{example}
\begin{definition}\label{normalform}
Fix a monomial ordering on $\C[\V]$. If $p\in\C[\V]$ then the normal form of $p$, denoted $[p]$, is the unique polynomial representative in $\C[z]/\bfI(\V)$ which contains no monomials in the ideal $\la \lt(\bfI(\V))\ra$.\footnote{See \cite[page 4]{CoxMau15} for details.} If $p,q\in\C[\V]$ then we define $p*q := [pq]$. 
\end{definition}
In \cite{CoxMau15} it was shown that under the mild geometric condition of distinct intersections with infinity that one can show the convergence of the transfinite diameter limit for algebraic varieties in an analogous way to $\C^M$. Their main results are summarised in the following two results. 
\begin{lemma}[Corollary 2.6, Lemma 2.7-2.10, Proposition 2.11, \cite{CoxMau15}] \label{coxmaulem}
Suppose that $\V$ is a pure $M$-dimensional affine algebraic variety with distinct intersections with infinity. For some $t\in\N$ sufficiently large there are polynomials $v_1,...v_d\in\C_{=t}[\V]$ satisfying the following properties.
\begin{itemize}
\item[(i)]{$v_i*v_i = x_M^{2t} + \sum_{k=1}^{M-1} x_kh_k + h_0$ with $\deg(h_k)\leq 2t-1$ for each $k=0, ..., M-1$.}
\item[(ii)]{$v_i*v_j = \sum_{k=1}^{M-1} x_kq_k + q_0$ if $i\neq j$ with $\deg(q_k)\leq 2t-1$ for each $k$.}
%\item[(iii)]{There is $\tilde v_i\in\C[x_0,x,y]/\la \bfI^h + \la z_0\ra\ra$ such that $\tilde v_i(\lambda_j) = \delta_{ij}$ where $\delta_{ij}$ is the Kronecker delta function and the canonical representation of $\tilde v_i(0,z_1,...,z_N)$ in $\C[\V]$ is $v_i$.}
\item[(iii)]{$\C[\V]$ is spanned by 
\begin{align*}
(\star)&\quad x^\beta x_M^l*y^\alpha, &&\beta\in\Z^{M-1}_{\geq0}, x^l_My^\alpha\in \mathcal{A}.\\
(\star\star)&\quad x^\beta x_M^l*v_i,&&\beta\in\Z^{M-1}_{\geq0}, l\geq 0, i=1,...,d.
\end{align*}
where $\mathcal{A} = \{x_M^ly^\alpha \not\in \la \lt(\bfI(\V))\ra: l+|\alpha|\leq t-1\}$. }
\end{itemize}
\end{lemma}
\begin{remark}\label{remark}
In Lemma \ref{deglem} we will prove that consideration of the normal forms in property (iii) of the previous lemma is unnecessary. As a consequence of this we can interpret the previous lemma as giving us a decomposition of $\C[\V]$ in the following way
\begin{align*}
\C[\V] =\bigoplus_{x^l_My^\alpha\in\mathcal{A}}x^l_My^\alpha\C[x_1,...,x_{M-1}]\,\,\bigoplus_{1\leq i\leq d} v_i\C[x_1,...,x_M]
\end{align*}
where the $v_i$'s have convenient algebraic properties. As a result of this the collection of $(\star)$ and $(\star\star)$ elements act like a monomial basis for $\C[\V]$ and this idea is utilised in the following theorem. Also note that $\mathcal{A}$ contains all monomials in $\C[\V]$ in the variables $x_M,y_1,...,y_{N-M}$ of degree at most $t-1$. 
\end{remark}
\begin{definition}
The elements from Lemma \ref{coxmaulem} will be called a cm-basis for $\C[\V]$.  
\end{definition}
\begin{theorem}[Theorem 5.2, \cite{CoxMau15}]
Let $\V$ be an affine algebraic variety with distinct intersections with infinity and $K$ a compact subset of $\V$. If $\cC$ %consists of the elements $(*)$ and $(**)$ from Lemma \ref{coxmaulem} 
is a cm-basis for $\C[\V]$ ordered by any graded ordering\footnote{A graded ordering is an ordering where the elements are ordered by total degree before any other ordering. In Cox-ma`u \cite{CoxMau15} a specific graded ordering is constructed for the proof, but their result holds for any graded ordering.} then the limit
\[d^{cm}(K) = \lim_{k\rightarrow\infty} \left(\|VDM[\cC_k]\|_{L^\infty(K)}\right)^{1/l_{k}}\]
exists. 
\end{theorem}
\subsection{Berman-Boucksom Transfinite Diameter}
As part of their breakthrough study of pluripotential on complex manifolds, Berman-Boucksom (\cite{Berman08, Berman10, Berman11}) showed that a transfinite diameter with respect to an orthonormal polynomial basis existed. The result we are interested in is rephrased in our context below.
\begin{theorem}[Corollary A, \cite{Berman10}]\label{bb} Suppose that $\V$ is an affine algebraic variety  of pure dimension $M$ and $E\subset\V$ a compact set. Let $\nu$ be a probability measure on $K$. Let $S_k$ be an $L^2(\nu)$-orthonormal basis for $\C_{k}[\V]$. Then for every subset $K\subset\V$ the limit
\[\log d^{bb}_\nu(K)=\lim_{k\rightarrow\infty} \frac{1}{l_{k}} \log \|VDM[S_k]\|_{L^\infty(K)}\]
exists.
\end{theorem} 
The original statement has $\frac{1}{l_{k}}$ replaced with $\frac{1}{kN_k}$. Relating these two formulations is purely combinatoric and the relation is clarified in Section 5. For our main theorem we need to find a particular probability measure $\nu$ which satisfies the hypothesis of Theorem \ref{bb}. 
\begin{definition}
Let $\V$ be a pure $M$-dimensional affine algebraic variety. We say $(x,y)\in\C^{M}\times\C^{N-M}$ is a Noether presentation for $\V$ if the following conditions are satisfied.
\begin{itemize}
\item[(i)]{$\C[x]$ is a Noether normalisation for $\C[\V]$.}
\item[(ii)]{There is a constant $C$ such that $\|y\| \leq C(1+\|x\|)$ for all $(x,y)\in \V$.}
\end{itemize}
\end{definition}
Coordinates of this type were introduced in \cite{Hart17} and studied further in \cite{HartMau18}. Coordinates satisfying only the second condition have been considered by other authors including Rudin \cite{Rudin68} and Zeriahi \cite{Zeriahi96}. Such coordinates can always be constructed and explicit methods of construction are given in \cite[Theorem 2.21]{Hart17} or \cite[Proposition 1.9]{HartMau18}. They are coordinates for which studying pluripotential theory on affine algebraic varieties in particularly convenient. An important observation is that, with respect to these coordinates, the variety is locally bounded in the $y$ variables i.e.\ $\|y\|\rightarrow\infty$ if and only if $\|x\|\rightarrow\infty$.\\\gap\\
Let $\log^+|z| = \max\{\log|z|, 0\}$ for all $z\in\V$. Recall the Lelong class of plurisubharmonic (psh) functions, 
\[\Lcal^+(\V) = \{u\in PSH(\V) : \log^+\|z\|+a \leq u(z) \leq \log^+\|z\|+b, \text{for some }a,b\in\R\}.\]
Noether presentations allow for the following calculation of mass result.
\begin{proposition}[Theorem 2.27 \cite{Hart17}, Theorem 3.6 \cite{HartMau18}]\label{masslem}
Suppose that $\V$ is a pure $M$-dimensional affine algebraic variety with $d$ sheets. If $u\in\Lcal^+(\V)$ then $\int_\V (\ddc u)^M = d$. 
\end{proposition}
Given a Noether presentation for an affine algebraic variety $\V$ let $T_\V = \{(x,y)\in\V : |x_j|\leq 1\}$. By the previous remark about the local boundedness of $\V$, it follows that $T_\V$ is compact. Recall that the logarithmic extremal function associated to a compact set $K$ is defined by
\[V_K(z) := \sup\{u(z) : u\in \Lcal^+(\V), u|_K\leq 0\}.\]
\begin{lemma}
Suppose that $\V$ is a pure $M$-dimensional affine algebraic variety with $d$ sheets and suppose we have a Noether presentation $(x,y)$. Let $\nu = (\ddc V_{T_\V})^M/d$ with respect to these coordinates, then $\nu$ is a probability measure on $\V$.
\end{lemma}
\begin{proof}
Consider $v(x,y) = \max\{\log^+|x_1|,...,\log^+|x_{M}|\}$. Clearly $v(x,y)\leq 0$ on $T_\V$ so $v(x,y)\leq V_{T_\V}$. But $v(x,y)$ is maximal outside of $T_\V$ and so any $u\in\Lcal^+(\V)$ with $u\leq V_{T_\V}$ also satisfies $u\leq v(x,y)$. It follows that $v=V_{T_\V}$ so $V_{T_\V}\in\Lcal^+(\V)$ and the first claim follows from Proposition \ref{masslem}. 
\end{proof}
\begin{definition}
Suppose that $\V$ is a pure $M$-dimensional affine algebraic variety with $d$ sheets and suppose we have a Noether presentation $(x,y)$. Let $\nu = (\ddc V_{T_\V})^M/d$ with respect to these coordinates. We say that the basis $\B$ that arises from performing Gram-Schmidt on $\M[\V]$ with respect to the $L^2(\nu)$-norm is a bb-basis for $\C[\V]$.
\end{definition}
%Remark to end of paper, \item{$\nu$ has the Bernstein-Markov property.} 
%Recall that $(\ Bernstein-Markov property 
%Section 3
\section{Relations between General Transfinite Diameters}
The following definition is motivated by the notion of a ring being finite over a subring. We want the same kind of structure but for a basis of the ring and subring.
\begin{definition}
Let $S[z]$ be a polynomial ring and $\B\subset S[z]$ some subset. We say that $E$ is a core for $\B$ if there is a pair $(A,t)$ where $A\subset S[z]$ is a finite set and $t\in\Z_{\geq0}$ such that
\[\{b\in\B : \deg(b) \geq t\} = \{ae : e\in E, a\in A\}\]
and $E$ is closed under multiplication.
\end{definition}
The following examples are fundamental. 
\begin{example}\label{ex1}%DIMENSION
Suppose that $\V=\bfV(y^2-x^2-1)$. Then $\C[x]$ is a Noether normalisation for $\C[\V]$. The monomials in $\C[\V]$ are given by $\{ x^n, x^ny : n\in\Z_{\geq0}\}$. Then $\M[x]$ is a core for $\M[\V]$ where $t=0$ and the finite set is $A=\{1,y\}$.
\end{example}
\begin{example}\label{ex2}
Suppose that $\V=\bfV(y^2-x^2-1)$. The $v$ polynomials from Lemma \ref{coxmaulem} are given by
\[v_1(x,y) = \frac{1}{\sqrt{2}}(y+x), \qquad v_2(x,y) = \frac{1}{\sqrt{2}}(y-x).\]
Then the set $\{1, x^nv_1(x,y), x^nv_2(x,y) : n\in\Z_{\geq0}\}$ has $\M[x]$ as a core where $t=1$ and the finite set is $A=\{v_1,v_2\}$. 
\end{example}
\begin{example}\label{ex3}
Suppose that $\V=\bfV(y^2-x_2^2-x_1^2-1)$. Then the $v_i$ polynomials from Lemma \ref{coxmaulem} are given by
\[v_1(x_1,x_2,y) = \frac{1}{\sqrt{2}}(y+x_2), \qquad v_2(x_1,x_2,y) = \frac{1}{\sqrt{2}}(y-x_2).\]
Then the corresponding spanning set from this Lemma is the set
\[\{x_1^n, x_1^nx_2^mv_1, x_1^nx_2^mv_2 : n,m\in\Z_{\geq 0}\}.\]
This set does not have a core. The first family of elements implies that a core could be $\M[x_1]$, however with this choice it is impossible to find an associated finite set since every element of $\M[x_2]$ is in the set. Similarly, $\M[x_1,x_2]$ cannot be a core since the corresponding finite set must be $\{1,v_1,v_2\}$ which would imply that $1\cdot\M[x_2]$ belongs to the set which is not true.
%To see this note that the finite set $A$ would need to be $\{1,v_1,v_2\}$. The family of elements associated to $1$ are generated by the monomials in $\C[x_1]$ so this implies that the core should be the monomials in $\C[x_1]$. However the family of elements associated to $v_1,v_2$ are generated by the monomials in $\C[x_1,x_2]$ which is different to the family of elements associated to 1.  
\end{example}
\begin{definition}\label{compliant}
Let $S[z]$ be a polynomial ring. Suppose that $\B$ and $\cC$ are subsets of $S[z]$. We say that $\B$ and $\cC$ are compliant if there is a set $E$ which is a core for $\B-\cC$ and $\cC-\B$. If in addition $\B$ and $\cC$ are bases for $S[z]$ then we say that $\B$ and $\cC$ are compliant bases.
\end{definition}
\begin{example}\label{ex4}
Let $\B$ be the basis from Example \ref{ex1} and $\cC$ the spanning set from Example \ref{ex2}. Then  \[\cC-\B = \{x^nv_1(x,y), x^nv_2(x,y) : n\in\Z_{\geq0}\},\]
and
\[\B-\cC = \{x\cdot x^n, yx^n : n\in\Z_{\geq0}\}.\]
Then $\B$ and $\cC$ are compliant with core $\M[x]$. The finite set $A$ for $\cC-\B$ is $\{v_1,v_2\}$ while the finite set for $\B-\cC$ is $\{x,y\}$ and $t\geq 1$ for both. 
\end{example}
\begin{example}\label{ex5}
Now consider $\V=\bfV(y^2-x_2^2-x_1^2-1)$ and let $\cC$ be the spanning set from Example \ref{ex3} and let $\B$ be $\M[\V]$ i.e.\ $\{1, x_1^nx_2^m, yx_1^nx_2^m : n,m\in\Z_{\geq0}\}$. Then
\[\cC-\B = \{x_1^nx_2^mv_1, x_1^nx_2^mv_2 : n,m\in\Z_{\geq 0}\},\]
and
\[\B-\cC = \{x_2\cdot x_1^nx_2^m, yx_1^nx_2^m : n,m\in\Z_{\geq0}\}.\]
Then $\B$ and $\cC$ are compliant with core given by $\M[x_1,x_2]$. The finite set $A$ for $\cC-\B$ is $\{v_1,v_2\}$ while the finite set for $\B-\cC$ is $\{x_2,y\}$ and $t\geq 1$ for both.
\end{example}
\begin{example}
For an example of bases which are not compliant, let $\B$ be the monomials $\{z^{n}:n\in\Z_{\geq0},z\in\C\}$. Let $\cC$ be the scaled monomials under the transformation $z\mapsto rz$, i.e. $\cC = \{(rz)^n : n\in\Z_{\geq0},z\in\C\}$. Then $\B-\cC=\B$ and $\B$ has core given by $\B$, while $\cC-\B=\cC$ and $\cC$ has core given by $\cC$. Since these cores are different it follows that $\B$ and $\cC$ are not compliant. 
 \end{example}
The following structural lemma is formally obvious, but will be useful when dealing with specific examples. 
\begin{lemma}\label{commoncore}
Let $S[z]$ be a polynomial ring. Suppose that $\B$ and $\cC$ are bases for $S[z]$ with each with core $E$. Then $\B$ and $\cC$ are compliant bases for $S[z]$. 
\end{lemma}
\begin{proof}
Suppose that $t$ is sufficiently large so that
\begin{align*}
\{b\in \B : \deg(b)\geq t\} &= \{ae : e\in E, a\in A_\B\}\\
\{c\in \cC : \deg(c)\geq t\} &= \{ae : e\in E, a\in A_\cC\}.
\end{align*}
Any element in $\B-\cC$ of degree $\geq t$ (resp.\ $\cC-\B$) has the form $ae, a\in A_\B$ (resp.\ $ae, a\in A_\cC$). It follows that $\B-\cC$ has core $E$ and $\cC-\B$ has core $E$ as required.
\end{proof}
Before proving our main result we need the following counting lemma. 
\begin{lemma}\label{countinglemma}
Let $\V$ be an affine algebraic variety of pure dimension $M$. Then
\[\lim_{k\rightarrow\infty} \frac{N_k}{l_k} =0.\]
\end{lemma}
\begin{proof}
Since $N_k$ and $l_k$ are preserved under linear transformations, we may assume that $\C[x]$, $x\in\C^M$ is a Noether normalisation for $\C[\V]$ %(via a linear transformation if necessary). 
We will utilise the following well known calculations (e.g.\ \cite[page 36]{LevNotes}) to derive the result; 
\[l^x_k = \frac{Mk}{M+1}N_k^x,\]
and so
\[\lim_{k\rightarrow\infty} \frac{N_k^x}{l_k^x} = \lim_{k\rightarrow\infty} \frac{M+1}{Mk} = 0.\]
Since $\C[x]$ is a Noether normalisation for $\C[\V]$, there exists finitely many elements $f_j\in\C[\V]$ such that $\C[\V]=\bigoplus_{j\in J} f_j\C[x]$. If $n=|J|$ then 
\[N_k = \dim(\C_k[\V]) = \dim\left(\bigoplus_{j\in J} f_j \C_{k-\deg(f_j)}[x]\right) \leq n\dim(\C[x]) = nN^x_k.\]
Next, clearly $N_k^x \leq N_k$ so it follows that $l^x_k \leq l_k$. With these calculations we can estimate
\[
\lim_{k\rightarrow\infty} \frac{N_k}{l_k} \leq \lim_{k\rightarrow\infty} \frac{nN^x_k}{l^x_k} =0 
\]
as desired.
\end{proof}
The following theorem is the main result of this section. It shows that two bases being compliant is a sufficient condition for their corresponding transfinite diameters to be equal.
\begin{theorem}\label{mainthm}
Let $\V$ be an affine algebraic variety  of pure dimension $M$, $K\subset \V$ a compact set, and suppose that $\B$ and $\cC$ are compliant graded bases for $\C[\V]$. If the limit
\[\log d^\cC(K)=\lim_{k\rightarrow\infty} \frac{1}{l_k} \log\|VDM[\cC_k]\|_{L^\infty(K)}\]
exists, then
\[\log d^{\B}(K)=\lim_{k\rightarrow\infty} \frac{1}{l_k} \log\|VDM[\B_k]\|_{L^\infty(K)}\]
exists. Moreover, $d^\B(K) = d^{\cC}(K)$. 
\end{theorem}
The strategy is to prove Theorem \ref{mainthm} is to transform $VDM[\B]$ into $VDM[\cC]$ then measure the impact the scale factor has on the determinant calculation. If the growth of that scale factor is sufficiently slow then it will decay to 1 in the limit giving the result. 
\begin{proof}
We first study the row operations needed to transform a row of $VDM[\B_k]$ into a row of $VDM[\cC_k]$. Let $E$ be the common core of $\B-\cC$ and $\cC-\B$. Let $C$ be the associated finite subset of $\cC-\B$ (with respect to $E$), and let $B$ be the finite set associated to $\B-\cC$ (with respect to $E$). For each $c_i\in C$ there is a (finite) linear combination of $e_{ij}b_{ij}\in \B$ such that
\[c_i = \sum_{j} a_{ij}e_{ij}b_{ij}, \quad a_{ij}\in\C\bs\{0\}. \]  
Multiplying through by some $e\in E$ we obtain
\[c_ie = \left(\sum_{j} a_{ij}e_{ij}b_{ij}\right)e=\sum_{j} a_{ij}\left(e_{ij}b_{ij}e\right), \quad a_{ij}\in\C\bs\{0\}.\]
These linear combinations tell us the row operations necessary to transform a row in $VDM[\B_k]$ to a row in $VDM[\cC_k]$. Precisely, 
\begin{itemize}
\item to make a row of $c_i$ entries we add $a_{ij}$ times the $e_{ij}b_{ij}$ row, 
\item to make a row of $c_ie$ entries we add $a_{ij}$ times the $e_{ij}b_{ij}e$ row.
\end{itemize}
Let $I$ be the set of double-indices $\{(i,j) : a_{ij}\neq 0\}$ appearing in these calculations. From linear algebra we know that the only row operation which scales the absolute value of the determinant is multiplication of a row by a constant. The multiplications that are needed to transform $VDM[\B_k]$ into $VDM[\cC_k]$ must come from the set $\{a_{ij}:ij\in I\}$. To this end, let $m= \min\{|a_{ij}| : (i,j)\in I\}$ and $M=\max\{|a_{ij}| : (i,j)\in I\}$. Now we can estimate
\[ m^{N_k}\|VDM[\cC_k]\|_{L^\infty(K)} \leq \|VDM[\B_k]\|_{L^\infty(K)} \leq M^{N_k}\|VDM[\cC_k]\|_{L^\infty(K)}\]
where the factor $N_k$ is the number of rows in the matrices. Taking $1/l_k$ roots, the $\log$ of everything, and the limit as $k\rightarrow\infty$ we conclude, using Lemma \ref{countinglemma}, that
\begin{align*}\lim_{k\rightarrow\infty}\frac{1}{l_k}\log\|VDM[\cC_k]\|_{L^\infty(K)}
&=\lim_{k\rightarrow\infty}\frac{N_k}{l_k}\log m+\frac{1}{l_k}\log\|VDM[\cC_k]\|_{L^\infty(K)} \\
&\leq \lim_{k\rightarrow\infty}\|VDM[\B_k]\|_{L^\infty(K)} \\
&\leq \lim_{k\rightarrow\infty}\frac{{N_k}}{l_k}\log M+\frac{1}{l_k}\log\|VDM[\cC_k]\|_{L^\infty(K)}\\
&=\lim_{k\rightarrow\infty}\frac{1}{l_k}\log\|VDM[\cC_k]\|_{L^\infty(K)}.
\end{align*}
\end{proof}
\begin{remark}
Adapting this proof to more general settings is possible if an analogous version of Lemma \ref{countinglemma} is available in this setting. 
\end{remark}
\section{Relation between $d^{cm}(K)$ and $d^{bb}(K)$}
Theorem \ref{mainthm} provides a useful test to determine whether transfinite diameters with respect to two bases are equivalent. Our strategy to show that $d^{cm}(K)$ and $d^{bb}(K)$ are equivalent is to show that the cm-basis and bb-basis are compliant with $\M[\V]$ and hence equivalent. It would be possible to show that the cm-basis and bb-basis are compliant directly but the working necessary lends itself naturally to a comparison to $\M[\V]$. 
\subsection{cm-Basis is Compliant with $\M[\V]$}
We begin with a lemma which justifies the claim about the normal form asserted in Remark \ref{remark}
\begin{lemma}\label{deglem}
Let $\V$ be an affine algebraic variety of pure dimension $M$. Suppose that $\C[x]$ is a Noether normalisation for $\C[\V]$. If there is some $p\in \bfI(\V)$ with $\lt(p)=x^\beta y^\alpha$ then there is some $q\in\bfI(V)$ with $\lt(q)=y^\alpha$. In particular if $y^\alpha\in\C[\V]$ then $x^\beta*y^\alpha = x^\beta y^\alpha$ and we have the decomposition
\[ \C[\V] = \bigoplus_{\alpha\in A} y^\alpha \C[x]\]
for some finite collection of multi-indices $A$.
\end{lemma}%Grevlex?
\begin{proof}
By hypothesis there exists $g_i\in\bfI(\V)$ and $m_i\in\Z_{\geq0}$ such that $\lt(g_i) = y_i^{m_i}$. Without loss of generality assume that $m_i$ is minimal for each $i$. Since $\lt(g_i) \neq \lt(g_j)$ for $i\neq j$ it is clear that the set $\{g_1,...,g_{N-M}\}$ is linearly independent, it follows that $\dim(\bfV(g_1,...,g_{N-M}))=M$. Since each $g_i\in\bfI(\V)$ it follows that $\V\subset\bfV(g_1,...,g_{N-M})$. Suppose that $f\in \bfI(\V)$ and $f\not\in \la g_1,...,g_{N-M}\ra$. Then $\dim(\bfV(g_1,...,g_{N-M},f))=M-1$ and $  \V\subset\bfV(g_1,...,g_{N-M},f)$ but $\dim(\V)=M$, a contradiction, which implies that there is no element $f$. It follows that $\bfI(\V) \subset \la g_1,...,g_{N-M}\ra$ which implies that $\bfV(g_1,...,g_{N-M})\subset \V$ and hence $\V=\bfV(g_1,...,g_{N-M})$.\\\gap\\
From this result the claims in the lemma follow. Since $p\in\bfI(\V)$ we can write $p = \sum f_ig_i$ and $\lt(p)=\lt(\sum f_ig_i)$ and the first claim follows. If $y^\alpha\in\C[\V]$ then $\lt(g_i)$ does not divide $y^\alpha$ for any $i$, and since the normal form calculation relies on polynomial division, it follows that $x^\beta *y^\alpha=x^\beta y^\alpha$. The decomposition of $\C[\V]$ is immediate with this last result.
\end{proof}
\begin{remark}
This property is not true for $\C[\V]$ in general. Consider $\V=(z^2+y^2-xy+x+1, z^2+y^2+y+2)\subset\C^3$. Then $x,y\in\C[\V]$ but $x*y = y-x+1$. In the situation of Lemma \ref{deglem} we have $xy-y+x-1\in \bfI(\V)$ and it is clear by construction that there is no $q\in\bfI(\V)$ such that $\lt(q)=y$ since $y\in\C[\V]$. This is also evident from the decomposition 
\[\C[\V] = \C[x] \oplus \C[y] \oplus z\C[x] \oplus z\C[y].\]
% Note that this reduction in degree does not always occur, e.g.\ $\deg(y\cdot y) = \deg(y^2) = 2= \deg(y)+\deg(y)$.  
\end{remark}
\begin{proposition}\label{cmmcompliant}
Let $\V$ be an affine algebraic variety with distinct intersections with infinity, $\C[x]$ a Noether normalisation for $\C[\V]$ and $\cC$ the corresponding cm-basis ordered by any graded ordering. Then $\cC$ and $\M[\V]$ are compliant bases for $\C[\V]$.
\end{proposition}
\begin{proof}
By Lemma \ref{coxmaulem} it is easy to see that
\begin{align*}
\cC-\M[\V] &= \{x^\beta v_i : \beta\in\Z_{\geq0}^M, i=1,...,d\},\\
\M[\V]-\cC &= \{x^\beta x_M^l y^\alpha : \beta\in\Z^{M}_{\geq 0}, x_M^l y^\alpha \in \M[z], l+|\alpha|=t\}.
\end{align*}
By Lemma \ref{deglem} these elements are in normal form. It is easy to see that common core for these sets is $\M[x]=\{x^\beta:\beta\in\Z^M_{\geq0}\}$ and that the finite sets are $C=\{v_i : i=1,...,d\}$ and $B=\{x_M^ly^\alpha : l+|\alpha|=t\}$. %To see that $C\subset \Span\{b\in B\}$ (and conversely) observe that $\Span(\cC_{=t}-\M_{=t}[\V])=\Span(\M_{=t}[\V]-\cC_{=t})$ since both $\cC$ and $\M[\V]$ are bases for $\C[\V]$. But $\cC_{=t}-\M_{=t}[\V]=C$ and $\M_{=t}[\V]-\cC_{=t}=B$ which shows that this is true, and the proposition follows.
\end{proof}
\begin{corollary}
Let $\V$ be an affine algebraic variety with distinct intersections with infinity, $K\subset\V$ a compact set, and $\C[x]$ a Noether normalisation for $\C[\V]$. Then $d^{cm}(K)=d^{\M[\V]}(K)$. 
\end{corollary}
\begin{proof}
By Proposition \ref{cmmcompliant}, any cm-basis is compliant with $\M[\V]$. The result now follows from Theorem \ref{mainthm}.
\end{proof}
\subsection{bb-Basis is Compliant with $\M[\V]$} We will show that $\M[\V]$ and a bb-basis are compliant with core $\M[x]$. Before we can do this we need to study the structure of a bb-basis.
\begin{proposition}\label{cxon}
Suppose that $\V$ is a pure $M$-dimensional affine algebraic variety with $d$ sheets and suppose we have a Noether presentation $(x,y)$. Let $\nu = (\ddc V_{T_\V})^M/d$ with respect to these coordinates. Then the monomials $\M[x]$ are $L^2(\nu)$-orthonormal. 
\end{proposition}
\begin{proof}
Observe that for any $\alpha\in\Z_{\geq 0}^M$
\begin{align*}\int_\V x^{\alpha}\overline{x^\alpha}\,(\ddc V_{T_\V})^M &= \int_{\del T_\V} x^{\alpha}\overline{x^\alpha} d\mu \\
&=  \frac{1}{d}\sum_{i=1}^d\int_{|x_1|=...=|x_M|=1}|x_1|^{\alpha_1}...|x_M|^{\alpha_M}\,d\mu = \int 1\,d\mu = 1, 
\end{align*}
where $\mu$ is the normalised Lebesgue measure on $\{|x_1|=...=|x_M|=1\}$. Note we have used the standard fact that $(\ddc \max\{\log^+|x_1|,...,\log^+|x_M|)^M$ is the normalised current of integration on $\{|x_1|=...=|x_M|=1\}$. Similarly, if $\alpha,\beta\in\Z^M_{\geq 0}$ with $\alpha\neq\beta$ then at least one of $\alpha_j-\beta_j\neq 0$, without loss of generality assume that this is the case for $\alpha_1-\beta_1$. It is standard that in this case
\begin{align*}
\int_{0}^{2\pi} e^{i\theta(\alpha_1-\beta_1)}\,d\theta = 0.
\end{align*}
Using Fubini's theorem and the fact about $(\ddc \max\{\log^+|x_1|,...,\log^+|x_M|)^M$, we calculate that
\begin{align*}
\int_\V x^{\alpha}\overline{x^\beta}\,(\ddc V_{T_\V})^M &= \int_{\del T_\V} x^{\alpha}\overline{x^\beta} d\nu = \frac{1}{d}\sum_{i=1}^d \int_{|x_1|=...=|x_M|=1}\,x^{\alpha}\overline{x^\beta} d\mu\\
&= \int_{\theta_1=0}^{2\pi} ... \int_{\theta_M=0}^{2\pi} e^{i\theta_1(\alpha_1-\beta_1)}\cdot...\cdot e^{i\theta_M(\alpha_M-\beta_M)}\,d\theta_M...d\theta_1\\
&=\underbrace{\int_{\theta_1=0}^{2\pi}e^{i\theta_1(\alpha_1-\beta_1)}\,d\theta_1}_{=0}\cdot\, ...\cdot\int_{\theta_M=0}^{2\pi}e^{i\theta_M(\alpha_M-\beta_M)}\,d\theta_M\\
&=0
\end{align*}
which proves the claim.
\end{proof} 
\begin{proposition}\label{bbstructure}
Let $\V$ be an affine algebraic variety of pure dimension $M$ with Noether presentation $(x,y)$ and probability measure $\nu=(\ddc V_{T_\V})^M/d$. Suppose that 
\[\C[\V]=\bigoplus_{\alpha\in A}y^\alpha\C[x].\]
Let $B_A$ be the elements of $\{y^\alpha:\alpha\in A\}$ after Gram-Schmidt with respect to $L^2(\nu)$. Then a bb-basis $\B$ for $\C[\V]$ is given by 
\[\B = \{x^\beta f_j : x^\beta\in \M[x], f_j\in B_A\},\]
i.e.\ $\M[x]$ is a core for $\B$.
\end{proposition}
\begin{proof}
There are two things to check; that $\B$ is an orthonormal set and that it is a spanning set for $\C[\V]$. The normalisation condition can be checked by calculating directly that
\begin{align*}
\la x^\beta f_j, x^\beta f_j\ra_{L^2(\nu)} &= \left(\int_{T_\V} |x|^{2\beta}|f_j|^2\,d\nu\right)^\frac{1}{2}=\left(\int_{T_\V} |f_j|^2\,d\nu\right)^\frac{1}{2}=1. 
\end{align*}
The second equality follows from by Proposition \ref{cxon}, while the last equality follows from the definition of $f_j$. The orthogonality of the set can also be checked by calculation in the following way. First consider $x^\beta f_j$ and $x^\alpha f_i$ with $i\neq j$ (we allow the possibility that $\alpha=\beta$.) Since each $|x_k|=1$ on $T_\V$ we can estimate
\begin{align*}
|\la x^\beta f_j, x^\gamma f_i\ra_{L^2(\nu)}| &= \left|\int_{T_\V} x^\beta \overline{x^\gamma} f_j\overline{f_i}\,d\nu\right|^\frac{1}{2}\leq \left|\int_{T_\V} f_j\overline{f_i}\,d\nu\right|^\frac{1}{2}=0. 
\end{align*}
And it follows that $x^\beta f_j$ and $x^\gamma f_i$ are orthogonal. Now suppose that $i=j$ and $\alpha\neq\beta$. Since $T_\V$ is compact, each $|f_j|$ is bounded by some constant $M$ independent of $j$. Then we estimate
\begin{align*}|\la x^\beta f_j, x^\gamma f_j\ra_{L^2(\nu)}| &= \left|\int_{T_\V} x^\beta \overline{x^\gamma} |f_j|^2\,d\nu\right|^\frac{1}{2}\\
&\leq M\left|\int_{T_\V} x^\beta \overline{x^\gamma}\,d\nu\right|^\frac{1}{2}=0, \end{align*}
by Proposition \ref{cxon}. This shows that the set is orthonormal. \\\gap\\
To show that $\B$ is a basis for $\C[\V]$ observe that $|B_A| = |A|$ and that the elements of $B_A$ are linearly independent. It follows that for any $k\geq a=\max\{|\alpha|\}$ that \[\dim(\C_k[\V]) = \dim\{x^\beta f_j : \deg(x^\beta f_j)\leq k, x^\beta\in\M[x], f_j\in B_A\}).\]
Since the elements of $B_A$ are linear combinations of elements from $\C[\V]$, the above equality in dimensions implies that $\Span\{b\in\B\} = \C[\V]$ as required.
%it is sufficient to show that any monomial in $\C[\V]\bs(\C_{a}[\V]\backslash \C[x])$ can be formed by a linear combination of elements from $\B$. Next observe that elements in $\C[\V]\bs(\C_{a}[\V]\backslash \C[x])$ have the form $x^\beta\cdot x^{\gamma}y^\alpha$ where $\alpha\in A$ and $|\gamma|+|\alpha|\leq a$. Since $x^{\gamma}y^\alpha\in \C_{ a}[\V]$ there is a linear combination of elements in $\{f_j\}_{1\leq j\leq t}$ which is equal to $x^{\gamma}y^\alpha$. From this observation it follows that $x^\beta\cdot x^{\gamma}y^\alpha=x^\beta\times(\text{a linear combination of }\{f_j\})$ and so $\B$ is a basis for $\C[\V]$. 
\end{proof}
\begin{proposition}\label{bbcompliant}
Let $\V$ be an affine algebraic variety of pure dimension $M$ with Noether presentation $(x,y)$. Let $\B$ be the elements of a bb-basis for $\V$. Then $\B$ and $\M[\V]$ are compliant bases for $\C[\V]$. 
\end{proposition}
\begin{proof}
From Propositions \ref{cxon} and \ref{bbstructure} we conclude that the elements which are not $L^2(\nu)$-orthonormal in $\M[\V]$ are $\{x^\beta y^{\alpha}:\alpha\in A, \alpha\neq0, \beta\in\Z^{M}_{\geq0}\}$. Again by Proposition \ref{bbstructure} it is sufficient to study the orthonormalisations of the $y^\alpha$ alone. To this end, write $[y^\alpha]_{\nu}$ to be the corresponding orthonormal elements (i.e.\ those obtained by using the Gram-Schmidt algorithm on $\M[\V]$). Let
\[B_A = \{[y^\alpha]_{\nu} : \alpha\in A\}.\]We make the following structural observations
\begin{align*}
\B &= \{ x^\beta f_j : f_j\in B_A, \beta\in\Z^{M}_{\geq 0}\}\\
\M[\V]&=\{x^\beta y^\alpha : \alpha\in A, \beta\in\Z^M_{\geq0}\}.
\end{align*}
Clearly both $\B$ and $\M[\V]$ have core $\M[x]$ and so by Lemma \ref{commoncore} $\B$ and $\M[\V]$ are compliant.
\end{proof}
\begin{corollary}
Let $\V$ be an affine algebraic variety of pure dimension $M$ with Noether presentation $(x,y)$ and $K\subset\V$ a compact set. Then if $\B$ is a bb-basis for $\V$ we have
\[d^{bb}_\nu(K) = d^{\M[\V]}(K).\]
\end{corollary}
\begin{proof}
By Proposition \ref{cmmcompliant}, a bb-basis is compliant with $\M[\V]$. The result now follows from Theorem \ref{mainthm}.
\end{proof}
\begin{corollary}\label{propcmbb}
Let $\V$ be an affine algebraic variety of pure dimension $M$ with distinct intersections with infinity and Noether presentation $(x,y)$. If $K\subset \V$ is a compact set then
\[d^{cm}(K) = d^{bb}_\nu(K).\]
\end{corollary}
\section{Adaption to the Original Statement of Berman-Boucksom}
The original statement of Berman-Boucksom had the definition of the transfinite diameter being
\begin{equation}\lim_{k\rightarrow\infty}\frac{1}{kN_k}\log\|VDM[S_k]\|_{L^{\infty}(K)}.\label{bboriginal}
\end{equation}
This statement is more convenient in their setting despite the unusual denominator. To relate equation (\ref{bboriginal}) to Corollary \ref{propcmbb} it is simply a matter of comparing the growth of $kN_k$ and $l_k$. The estimate we need to derive is not immediately obvious so we consider the following example to motivate the derivation. 
\begin{example}Suppose that $(x,y)\in \C^M\times\C^2$, $A=\{(0,0), (1,0), (0,1), (2,0), (1,1), (0,2)\}$, and $\C[\V]=\bigoplus_{\alpha\in A} y^\alpha \C[x]$ then we can count the elements in $N_k$ via the following table
\begin{align*}
\begin{array}{rlllllllll}
\text{Terms of degree}& & \deg 0& \deg 1 & \deg 2 & \deg 3 & \deg 4 & ... & \deg k\\
\hline \text{Terms of the form }1\times\M[x] & & N^x_{=0} & N_{=1}^x & N_{=2}^x & N_{=3}^x & N_4^x & ... & N_{=k}^x \\
\text{Terms of the form }y^{(1,0)}\times\M[x] & & & N^x_{=0} & N_{=1}^x & N_{=2}^x & N_{=3}^x & ... &N_{=k-1}^x \\
\text{Terms of the form }y^{(0,1)}\times\M[x] & & & N^x_{=0} & N_{=1}^x & N_{=2}^x & N_{=3}^x & ... &N_{=k-1}^x \\
\text{Terms of the form }y^{(2,0)}\times\M[x] & & & & N^x_{=0} & N_{=1}^x & N_{=2}^x & ... & N_{=k-2}^x \\
\text{Terms of the form }y^{(1,1)}\times\M[x] & & & & N^x_{=0} & N_{=1}^x & N_{=2}^x & ... & N_{=k-2}^x \\
\text{Terms of the form }y^{(0,2)}\times\M[x] & & & & N^x_{=0} & N_{=1}^x & N_{=2}^x & ... & N_{=k-2}^x 
\end{array}
\end{align*}
Since $N_k=\sum_{j\leq k} N_{=j}$, observe that $6N_{1}^x\leq N_{3} \leq 6N_{3}^x$. In general we have $6N_{ k-2}^x\leq N_k \leq 6N_k^x$ where the $-2$ term comes from $a=\max\{|\alpha|:\alpha\in A\}$ and the 6 is equal to $n=|A|$. The following lemma formalises this observation.
\end{example}
\begin{lemma}\label{ineqlem}
Suppose that $\C[\V]=\bigoplus_{\alpha\in A} y^\alpha\C[x]$. Then 
\[nN^x_{ k-a}\leq N_k \leq nN^x_{k}\]
and
\[nl^x_{ k-a}\leq l_k \leq nl_k^x\]
where $n=|A|$ and $a=\max\{|\alpha|:\alpha\in A\}$. 
\end{lemma}
\begin{proof}
Let $N^\alpha_{ k}$ be the number of monomials of degree at most $k$ which are of the form $y^\alpha x^\beta$ where $\beta\in\Z_{\geq0}^{M}$. Then $N^\alpha_{ k}= N^x_{ k-|\alpha|}$. As $a\geq |\alpha|$ for any $\alpha\in A$, clearly 
\[ N^x_{ k-a} \leq N^\alpha_{ k} \leq N^x_{ k} \]
for any $\alpha$. Summing over all $n$ possibilities for $\alpha$ we obtain
\[nN^x_{ k-a} \leq N_k \leq nN_k^x\]
as desired. The argument for $l_k$ is identical.
\end{proof}
\begin{lemma}
$\displaystyle\lim_{k\rightarrow\infty} \frac{kN_k}{l_k} = \frac{M+1}{M}.$
\end{lemma}
\begin{proof}
%First we derive an estimate for $N_k$. Recall that for $k\geq a$
%\[\C_{ k}[\V] = \bigoplus_{\alpha\in A}y^\alpha\C_{ k-|\alpha|}[x].\]
If $a=\max\{|\alpha|:\alpha\in A\}$ and $n=|A|$ then by Lemma \ref{ineqlem}
\[ nN_{ k-a}^x \leq N_k \leq nN_k^x \]
 and 
 \[nl_{a-k}^x \leq l_k \leq nl_k^x.\]
 From these estimates and the identity $l_k^x = \frac{M}{M+1} kN_k^x$ it follows that
 \begin{equation}\frac{M+1}{M} \frac{N_{ k-a}^x}{N_k^x}= \frac{knN_{ k-a}^x}{nl_k^x} \leq \frac{kN_k}{l_k} \leq \frac{knN_k^x}{nl^{x}_{ k-a}} = \frac{M+1}{M}\frac{N_k^x}{N_{ k-a}^x}.\label{estim}\end{equation}
 Now we compute
\begin{align*}\frac{N_k^x}{N_{ k-a}^x} &= \frac{\begin{pmatrix} M+k \\ M \end{pmatrix}}{\begin{pmatrix} M+k-a \\ M \end{pmatrix}} = \frac{(M+k)!}{M!k!} \frac{M!(k-a)!}{(M+k-a)!} \\
&= \frac{(M+k)\cdot...\cdot(M+k-a+1)}{k \cdot ... \cdot (k-a+1)} \\
&= \underbrace{\left(\frac{M}{k}+1\right)\cdot...\cdot\left(\frac{M}{k-a+1}+1\right)}_{a\text{ multiplications}}\\
&\longrightarrow 1\quad\text{ as }k\rightarrow \infty.
\end{align*}
This calculation in conjunction with the estimate (\ref{estim}) yields
\[\lim_{k\rightarrow\infty} \frac{kN_k}{l_k} = \frac{M+1}{M}.\]
\end{proof}
\begin{corollary}\label{bbrelated}
\[\frac{M+1}{M}\log d^{bb}_\nu(K) = \lim_{k\rightarrow\infty}\frac{1}{kN_k}\log\|VDM[S_k]\|_{L^{\infty}(K)}\]
\end{corollary}
\section{An Example}
Let $\V = \bfV(y^2-x^2-1)$, by Example \ref{example2.2} this variety has distinct intersections with infinity. \\\gap\\
\textit{Cox-ma`u Transfinite Diameter.}\\ To construct the $v_i$ polynomials from Lemma \ref{coxmaulem} we use techniques from \cite{Mau11}. To this end, up to a scale factor, $v_1$ and $v_2$ come from the factors of the top degree homogeneous part of $\V$ i.e.\ $y^2-x^2$. We can set $v_1 = \frac{1}{\sqrt{2}}(y-x)$ and $v_2=\frac{1}{\sqrt{2}}(y+x)$. It is easily checked that these polynomials satisfy the required conditions:
\begin{align*}
v_1^2 &= \frac{1}{2}y^2-\sqrt{2}yx+\frac{1}{2}x^2 = -\sqrt{2}yx+x^2+1 = x^2-\sqrt{2}(v_1+v_2)x+\frac{1}{2}\\
v_2^2 &= \frac{1}{2}y^2+\sqrt{2}yx+\frac{1}{2}x^2 = \sqrt{2}yx+x^2+1 = x^2+\sqrt{2}(v_1+v_2)x+\frac{1}{2}\\
v_1v_2&= \frac{1}{2}(y^2-x^2) = \frac{1}{2}.
\end{align*}
The corresponding cm-basis for $\C[\V]$ is $\{1, x^m v_1, x^m v_2 : m\in\Z_{\geq0}\}$ and a typical Vandermonde determinant takes the form
\begin{align*}
&VDM_{\zeta_1,...,\zeta_t}[\cC_k] = \det\begin{pmatrix} 1 & 1 & ... & 1 \\
v_1(\zeta_1) &  v_1(\zeta_2) & ... & v_1(\zeta_t) \\
v_2(\zeta_1) & v_2(\zeta_2) & ... & v_2(\zeta_t)  \\
\zeta_{1x}\cdot v_1(\zeta_1) & \zeta_{2x}\cdot v_1(\zeta_2) & ... & \zeta_{tx}\cdot v_1(\zeta_t)\\
\zeta_{1x}\cdot v_2(\zeta_1) & \zeta_{2x}\cdot v_2(\zeta_2) & ... & \zeta_{tx}\cdot v_2(\zeta_t)\\
\vdots & \vdots & \ddots & \vdots \\
\zeta_{1x}^{k-1}\cdot v_2(\zeta_1) & \zeta_{2x}^{k-1}\cdot v_2(\zeta_2) & ... & \zeta_{tx}^{k-1}\cdot v_2(\zeta_t)\end{pmatrix}
\end{align*}
where we have written $\zeta_{ix}$ for the $x$-coordinate of $\zeta_i$ and $t$ is the number of elements in $\cC_k$. Observe that $v_2+v_1 = \frac{1}{\sqrt{2}} y$ and $v_2-v_1 = \frac{1}{\sqrt{2}} x$. Thus to obtain the usual monomials $1, x, y, x^2, xy, ...$ in this determinant we add rows together to obtain the determinant
\[\det\begin{pmatrix} 1 & 1 & ... & 1 \\
\frac{1}{\sqrt{2}} \zeta_{1x} &  \frac{1}{\sqrt{2}} \zeta_{2x} & ... & \frac{1}{\sqrt{2}} \zeta_{tx} \\
\frac{1}{\sqrt{2}} \zeta_{1y} & \frac{1}{\sqrt{2}} \zeta_{2y}  & ... & \frac{1}{\sqrt{2}} \zeta_{ty}   \\
\frac{1}{\sqrt{2}}\zeta_{1x}^2 & \frac{1}{\sqrt{2}}\zeta_{2x}^2 & ... & \frac{1}{\sqrt{2}}\zeta_{tx}^2 \\
\frac{1}{\sqrt{2}}\zeta_{1x}\zeta_{1y} & \frac{1}{\sqrt{2}}\zeta_{2x}\zeta_{2y} & ... & \frac{1}{\sqrt{2}}\zeta_{tx}\zeta_{ty}\\
\vdots & \vdots & \ddots & \vdots \\
\frac{1}{\sqrt{2}}\zeta_{1x}^{k-1}\zeta_{1y} & \frac{1}{\sqrt{2}}\zeta_{2x}^{k-1}\zeta_{2y} & ... & \frac{1}{\sqrt{2}}\zeta_{tx}^{k-1}\zeta_{ty}\end{pmatrix}\]
Factoring out the $\frac{1}{\sqrt{2}}$ terms we obtain
\[\left(\frac{1}{\sqrt{2}}\right)^{t-1}\det\begin{pmatrix} 1 & 1 & ... & 1 \\
 \zeta_{1x} &   \zeta_{2x} & ... &  \zeta_{tx} \\
 \zeta_{1y} &  \zeta_{2y}  & ... &  \zeta_{ty}   \\
\zeta_{1x}^2 & \zeta_{2x}^2 & ... & \zeta_{tx}^2 \\
\zeta_{1x}\zeta_{1y} & \zeta_{2x}\zeta_{2y} & ... & \zeta_{tx}\zeta_{ty}\\
\vdots & \vdots & \ddots & \vdots \\
\zeta_{1x}^{k-1}\zeta_{1y} & \zeta_{2x}^{k-1}\zeta_{2y} & ... & \zeta_{tx}^{k-1}\zeta_{ty}\end{pmatrix} = \left(\frac{1}{\sqrt{2}}\right)^{t-1} VDM_{\zeta_1,...,\zeta_t}[\M_k[\V]].\]
We can calculate that $t = 2k-1$ and $l_k = \frac{k(2k-1)}{2}$. Clearly $\lim_{k\rightarrow\infty} t/l_k = 0$. It then follows that
\begin{align*}\lim_{k\rightarrow\infty} \frac{1}{l_k}\log\|VDM[\cC_k]\|_{L^\infty(K)} &= \lim_{k\rightarrow\infty} \frac{1}{l_k}\log\|VDM[\M_k[\V]]\|_{L^\infty(K)}-\frac{t-1}{l_k}\log\sqrt{2} \\
&= \lim_{k\rightarrow\infty} \frac{1}{l_k}\log\|VDM[\M_k[\V]]\|_{L^\infty(K)}\end{align*}
as expected.\\\gap\\
{\it Berman-Boucksom Transfinite Diameter}\\
We have $T_\V = \{(x,y)\in V : |x|\leq 1\}$ and $\nu = \frac{1}{2}\ddc V_{T_\V} = \frac{1}{2}\ddc\log^+|x|$. It is well known that $\ddc\log^+|x|$ is equal to the normalised Lebesgue measure $\mu$ on the unit circle in the classical case (e.g.\ \cite[Theorem 3.7.4]{Ransford}). Observe that $\V$ is a two sheeted covering over $x$, hence the summation sign in what is to come. We perform Gram-Schmidt on the monomials $1, x, y, x^2, yx, ...$ with respect to $\nu$ to obtain the orthonormal basis $S$.
\begin{align*}
\la 1, 1\ra_{L^2(\nu)}^2 &= \int_{\V} 1\,d\nu = \frac{1}{2}\sum_{i=1}^2\int_{|x|=1} 1\,d\mu = \frac{2}{2}=1&\longrightarrow& 1\in S\\
\la 1, x\ra_{L^2(\nu)}^2 &= \int_{\V} \overline{x}\,d\nu = \frac{1}{2}\sum_{i=1}^2\int_{0}^{2\pi} e^{-i\theta} \,\frac{d\theta}{2\pi} = 0  && \\
\la x, x\ra_{L^2(\nu)}^2 &= \int_{\V} x\overline{x}\,d\nu = \frac{1}{2}\sum_{i=1}^2\int_{0}^{2\pi} 1 \,\frac{d\theta}{2\pi} = 1 &\longrightarrow&x\in S\\
\la 1, y\ra_{L^2(\nu)}^2 &= \int_{\V} \overline{y}\,d\nu = \frac{1}{2}\sum_{i=1}^2\int_{-\pi}^{\pi}\overline{\sqrt{1+e^{2i\theta}}}\,d\theta = 0 &&\\
\la x, y\ra_{L^2(\nu)}^2 &= \int_{\V} x\overline{y}\,d\nu = \frac{1}{2}\sum_{i=1}^2\int_{-\pi}^{\pi}e^{i\theta}\overline{\sqrt{1+e^{2i\theta}}}\,d\theta = 0 &&\\
\la y, y\ra_{L^2(\nu)}^2 &= \int_{\V} y\overline{y}\,d\nu = \frac{1}{2}\sum_{i=1}^2\int_{-\pi}^{\pi}|1+e^{-2i\theta}|\,d\theta = \frac{4}{{\pi}} &\longrightarrow& \frac{\sqrt\pi}{2}y \in S\\
&\,\,\,\vdots &&
\end{align*}
Continuing in this way one obtains the orthonormal basis $S=\{x^m, \frac{\sqrt{\pi}}{2}yx^m: m\in\Z_{\geq 0}\}$. Similar to the previous part, we can deduce that
\begin{align*}\lim_{k\rightarrow\infty} \frac{1}{l_k}\log\|VDM[S_k]\|_{L^\infty(K)} &= \lim_{k\rightarrow\infty} \frac{1}{l_k}\log\|VDM[\M_k[\V]]\|_{L^\infty(K)}-\frac{t-1}{l_k}\log\frac{\sqrt{\pi}}{2}\\
&=\lim_{k\rightarrow\infty} \frac{1}{l_k}\log\|VDM[\M_k[\V]]\|_{L^\infty(K)}\end{align*}
as expected.
\bibliographystyle{plain}

\end{document}